\documentclass[11pt, oneside]{amsart}
\RequirePackage{amsmath}
\usepackage{amsthm}

\usepackage[utf8]{inputenc}
\usepackage{hyperref}
\usepackage{url}

\usepackage{nomencl}
\makenomenclature

\usepackage{mathtools}
\usepackage{enumerate}
\usepackage{cite}
\usepackage[shortlabels]{enumitem}
\usepackage{tikz-cd}
\usepackage{stmaryrd}

\usepackage{latexsym}
\usepackage{amsmath,amssymb,amsthm,amscd}
\usepackage{color}
\usepackage[nolabel, final]{showlabels}
\title[Equivariant vector bundles on $X_E$]{On equivariant vector bundles on the Fargues--Fontaine curve over a finite extension}
\author{Rustam Steingart}
\address{Ruprecht-Karls-Universität Heidelberg,
	Mathematisches Institut,Im Neuenheimer Feld 205, D-69120 Heidelberg}
\email{rsteingart@mathi.uni-heidelberg.de}
\date{\today}

\theoremstyle{plain}
\newtheorem{thm}{Theorem}[section]
\newtheorem{lem}[thm]{Lemma}
\newtheorem{rem}[thm]{Remark}
\newtheorem{prop}[thm]{Proposition}

\newtheorem{ex}[thm]{Example}
\newtheorem*{cor*}{Corollary}
\newtheorem{introtheorem}{Theorem}

\theoremstyle{definition}
\newtheorem{defn}[thm]{Definition}

\newcommand{\NN}{\mathbb{N}}

\newcommand{\Gal}{\operatorname{Gal}}
\newcommand{\Hom}{\operatorname{Hom}}

\newcommand{\ZZ}{\mathbb{Z}}

\newcommand{\QQ}{\mathbb{Q}}

\newcommand{\RR}{\mathbb{R}}
\newcommand{\CC}{\mathbb{C}}

\newcommand{\Div}{\mathrm{Div}}
\newcommand{\ord}{\mathrm{ord}}

\newcommand{\fB}{\mathbf{B}}

\newcommand{\cR}{\mathcal{R}}

\newcommand{\id}{\operatorname{id}}

\newcommand{\Bdr}{\fB_{\mathrm{dR}}}
\newcommand{\Bcris}{\fB_{\mathrm{cris}}}

\DeclarePairedDelimiter\abs{\lvert}{\rvert}%
\begin{document}
	\maketitle
	\begin{abstract}
	Let $K/E/\QQ_p$ be a tower of finite extensions with $E$ Galois. We relate the category of $G_K$-equivariant vector bundles on the Fargues--Fontaine curve with coefficients in $E$ with $E$-$G_K$-$B$-pairs and describe crystalline and de Rham objects in explicit terms. When $E$ is a proper extension, we give a new description of the category in terms of compatible tuples of $\fB_e$-modules, which allows us to compute Galois cohomology in terms of an explicit Čech complex which can serve as a replacement of the fundamental exact sequence. 
	\end{abstract}
	\section*{Introduction} 
	Let $E/\QQ_p$ be finite Galois, let $\pi_E$ be a uniformiser of $E$ and let us denote by $\varphi_E$ the $q = p^{f(E/\QQ_p)}$-Frobenius. 
	In \cite{farguesfontainecourbe} Fargues and Fontaine assign to a non-archimedean local field $E$ and a perfectoid field $F$ of characteristic $p$ a ``curve''\footnote{The curve $X_{F,E}$ is not of finite type over $E.$ It is a curve in a slightly more general sense than usual (cf. section \ref{sec:vectb}).}  $X_{F,E},$ as $\operatorname{Proj}(P_{F,E})$ with a certain graded ring $P_{F,E} = \oplus_{d\geq0} (B_{F,E}^+)^{\varphi_E=\pi_E^d},$ where $B_{F,E}$ is an extension of the ring of ramified Witt vectors $W(o_F)_E[1/\pi_E]$ equipped with an extension of the Frobenius operator $\varphi_E.$ This curve plays a fundamental role in $p$-adic Hodge Theory. When $F=\CC_p^\flat$ and $E=\QQ_p$ the curve $X=X_{\CC_p^\flat,\QQ_p}$ has a distinct point $\infty$ with finite $G_{\QQ_p}$-orbit. The complement of the point is isomorphic to the spectrum of  $\fB_e$ and the completed stalk at $\infty$ (resp. the completion of the function field at the valuation corresponding to $\infty$) are isomorphic to $\Bdr^+$ (resp. $\Bdr$). The fundamental exact sequence 
	$$0 \to \QQ_p \to \fB_e \oplus \Bdr^+ \to \Bdr \to 0$$ obtains a geometric interpretation in the form of $H^0(X,\mathcal{O}_X)=\QQ_p$ and  $H^1(X,\mathcal{O}_X)=0.$ 
	The category of continuous representations of an open subgroup $G_K\subset G_{\QQ_p}$ on finite dimensional $\QQ_p$-vector spaces can be embedded via $V \mapsto V\otimes_{\QQ_p}\mathcal{O}_X$ as a full subcategory of the category of locally free $G_K$-equivariant sheaves. 
	In the first part of the article we revisit the notion of $G_K$-$E$-$B$-pair (subsequently referred to as $B$-pair) introduced by Nakamura (building on work of Berger) and relate it to the notion of $G_K$-equivariant vector bundles on the Fargues--Fontaine curve $X_E:=X_{\CC_p^\flat,E}.$ While the case $E=\QQ_p$ is extensively covered in the literature, the case $E\neq \QQ_p$ becomes more subtle. For example, Pham in  \cite{pham2023prismaticfcrystalsanalyticcrystalline} studies a natural seeming notion of ``crystalline'' vector bundles, but the slope $0$ objects are equivalent to $E$-crystalline representations (crystalline and $E$-analytic meaning Hodge-Tate of weight $0$ outside of a fixed embedding $E \to \CC_p$) of Kisin and Ren,  a category which for $E\neq \QQ_p$ is smaller. We give a more refined definition which leads to the following equivalence.
	\begin{introtheorem}
		The following hold
		\begin{enumerate}
			\item The category of $G_K$-$E$-$B$-pairs is equivalent to the category of $G_K$-bundles on $X_E.$
			\item Under the above equivalence, crystalline, (resp. de Rham, resp. slope $0$ objects) correspond to one another. 
		\end{enumerate}
	\end{introtheorem}
	
	While the above Theorem is as expected, we also show that for $E\neq \QQ_p$ there are new phenomena which appear. 
	Geometrically, the curve $X_E \xrightarrow{p_E} X_{\QQ_p}$ has $[E:\QQ_p]$-points lying above the distinguished point $\infty$ corresponding to Fontaine's period $ t_\infty:=t_{\text{cyc}} .$ This allows us to puncture $X_E$ in multiple different ways, and hence allows us to consider a covering of $X_E$ by punctured curves which are given explicitly as 
	$$U_T := X\setminus T\cong \operatorname{Spec}\left(\left(B^+_E\left[\frac{1}{\prod_{x \in T}t_x}\right]\right)^{\varphi_E=1}\right),$$
	where $ T \subset S =  \operatorname{Spec}(\kappa(\infty))\times_{X_{\QQ_p}}X_E$ is a finite subset and for $x \in S$ we denote by $t_{x}$ the period corresponding to $x.$
	More precisely, one has that $\kappa(\infty) \cong \CC_p$ and $X_E = E \times_{\QQ_p}X_{\QQ_p}.$ Thus the points above $\infty \in X_{\QQ_p}$ are in bijection with the embeddings $E \to \CC_p.$ Fixing a uniformiser $\pi_E$ of $E,$ allows us to define a base point by considering the period $t_{\mathrm{LT}}$ attached to the Lubin--Tate character $\chi_{\mathrm{LT}}.$
    By general results due to Fargues and Fontaine, the closed points of the curve are in bijection with 	
	$((B_E^+)^{\varphi_E=\pi_E}\setminus \{0\})/E^\times.$ The other closed points in $S$ are obtained in a similar manner using embeddings $\sigma \colon E \to \CC_p.$ We warn at this point that they are \textit{not necessarily } the periods attached to $\sigma\circ \chi_{\mathrm{LT}}$ as those would be in the $\varphi_E= \sigma(\pi_E)$ Eigenspace. Instead one has to consider $o_{\breve{E}}=W(\overline{\mathbb{F}_q})_E$-multiples of those (cf. Proposition \ref{prop:elementetinftysigma}). 
	
	 We see that the ring $\mathcal{O}_{X_E}(U_T)$ resembles $\fB_e = \fB_{\text{cris}}^{\varphi_p=1} \cong B_{\QQ_p}^+[1/t_{\text{cyc}}]^{\varphi_p=1}$ from $p$-adic Hodge Theory and by choosing a suitable family of subsets $\mathfrak{L}$ of $S$ we can express the condition of being a locally free sheaf (resp. $G_K$-equivariant sheaf) in  explicit terms as a tuple of free $\mathcal{O}_{X}(U_T)$-modules (resp. with continuous $G_K$-action)  for $T \in \mathfrak{L}$ satisfying certain compatibilities, which we call $\fB_e$-tuples.
	\begin{introtheorem}

		\label{thm:equivalencetuples1}
		Let $\mathfrak{L}$ be a co-covering of $S,$ then the functor 
		$$\operatorname{Bun}_{X_E} \to \{\fB_e\text{-tuples indexed by } \mathfrak{L}\}$$
		sending $\mathcal{F}$ to $(\mathcal{F}(X\setminus T)_T)$ is an equivalence of categories. 
		If $K/E$ is a finite extension then the same holds for the  $G_K$-equivariant version. 
	\end{introtheorem}
	Here by co-covering we mean a family of subsets $\mathfrak{L}$ not containing $\emptyset$ such that $X=\bigcup_{T \in \mathfrak{L}} U_T.$
	Let us point out that, since $\abs{S} = [E:\QQ_p]$ such a co-covering only exists if $[E:\QQ_p]>1.$ 
	As a corollary of Theorem \ref{thm:equivalencetuples1}
	we can compute the continuous Galois cohomology of a $G_K$-$E$-$B$-pair in terms of the  Čech complex of the $\fB_e$-tuple explicitly,
	i.e., $$\mathbf{R}\Gamma_{\text{sheaf}}(X_E,\mathcal{F}) \simeq C^\bullet(\mathfrak{L},(U_T)_{T \in \mathfrak{L}}),$$
	where the right hand side denotes the Čech complex for the covering $U_T.$
	To illustrate the usefulness of this concept let us consider the following example: 
	Consider two embeddings $\sigma \neq \tau.$ Let $V \in \operatorname{Rep}_E(G_K).$
	We get a ``fundamental exact sequence'' of the form
	$$0 \to V \to B_E^+[1/t_\sigma]^{\varphi_E=1} \otimes V \times B_E^+[1/t_\sigma]^{\varphi_E=1} \otimes V  \to B_E^+[1/t_\sigma t_\tau]^{\varphi_E=1} \otimes V \to 0.$$
	Using known results from the theory of almost $\CC_p$-representations, one can show that the differentials in $C^\bullet(\mathfrak{L},(U_T)_{T \in \mathfrak{L}})$ are strict. This allows us to compute $G_K$-cohomology in terms of this modified fundamental exact sequence. 
	It is in practice very hard to prove that a complex has strict differentials and the theory of $\fB_e$-tuples provides many different such complexes. The objects are in a certain sense of a multivariate nature. We will show that there is a canonical functor from multivariate $(\varphi,\Gamma)$-modules of \cite{Berger2013} to $\fB_e$-tuples and it is not difficult to see that every crystalline object arises in this way. It would be interesting to investigate which conditions on the $(\varphi,\Gamma)$-module side make this functor fully faithful and whether it is essentially surjective, given that according to our Theorem 	\ref{thm:equivalencetuples1} we would want a category of multivariable $(\varphi,\Gamma)$-modules which is equivalent to our category of $G_K$-$\fB_e$-tuples. 
		\section*{Acknowledgements}
	This research was supported by Deutsche Forschungsgemeinschaft (DFG) - Project number 536703837, which allowed me to carry out my research at the UMPA of the ENS de Lyon. I would like to thank the institution and in particular Laurent Berger for his guidance and many fruitful discussions. I thank Dat Pham and Marvin Schneider for comments on an earlier draft. 
	
	\section{Preliminaries}

	We summarise the main classical view points: $B$-pairs, cyclotomic $(\varphi,\Gamma)$-modules and Lubin--Tate $(\varphi,\Gamma)$-modules. Before we proceed let us quickly comment on the roles of $E$ and $K.$
	The assumption $E/\QQ_p$ Galois is not strictly required, but makes the notation easier.
 Nakamura usually requires $E$ to contain a normal closure of $K$ and endows for example $E \otimes_{\QQ_p} \Bdr$ with the action given by tensoring with the trivial action on $E$ but then the decomposition 
 $E\otimes_{\QQ_p} \Bdr = \prod_{\sigma\colon K\to E} E\otimes_{K,\sigma} \Bdr$ is $G_K$-equivariant, while the decomposition 
 $E\otimes_{\QQ_p} \Bdr = \prod_{\sigma \colon E \to \CC_p} E\otimes_{E,\sigma}\Bdr$ is not, however, given the geometry of $X_E$ with $[E:\QQ_p]$-many points above $\infty$ we would prefer to work with the latter decomposition and view it as a representation at each completed stalk above $\infty.$ As a consequence we shall assume $E \subset K$ instead. 
It is possible to treat the situation $K \subset E$ by requiring that the collection of completed stalks is a $G_K$-semi-linear representation. So instead of considering an action at each completed stalk $\widehat{\mathcal{F}}_s$ at $s \in S$ one would replace this by an action on 
 $\iota^{-1} \mathcal{F}$ where $\iota \colon S \to X_E$ is the inclusion.
 
	\subsection{$B$-pairs} 
	Let $E/\QQ_p$ be finite Galois and let $K$ be an extension of $E.$  Let us denote $\Sigma_E:= \Hom(E,\Bdr^+)$ and fix one embedding $\sigma_0 \in \Sigma_E.$
	\begin{defn} \label{def:bpair}
		An \textbf{$E$-$B$-pair}  is a  pair $W=(W_e,W_{\mathrm{dR}}^+)$ consisting of a continuous finite free $E\otimes_{\QQ_p}\fB_e$-representation $W_e$  of $G_K$ together with a $G_K$-equivariant $\Bdr^+$-lattice $W_{\mathrm{dR}}^+\subseteq W_{\mathrm{dR}}:= \Bdr \otimes_{\fB_e} W_e.$ We define the \textbf{rank} of $W$ as $\operatorname{rank}(W):=\operatorname{rank}_{E\otimes_{\QQ_p}\fB_e}W_e.$
		We denote by $C^\bullet(W)$ the complex $[W_e \oplus W_{\mathrm{dR}}^+ \to W_{\mathrm{dR}}]$ concentrated in degrees $[0,1]$ and 
		we define the \textbf{Galois cohomology} of $W$  as 
		$$\mathbf{R}\Gamma(G_K,W):=\mathbf{R}\Gamma_{cts}(G_K,C^\bullet(W))$$ and write $H^i(G_K,W)$ for the cohomology groups.  
		
		We denote by $\mathcal{BP}_E$ the category of $E$-$B$-pairs with the obvious notion of morphisms. 
		A morphism $W \to W'$ of $B$-pairs is called \textbf{strict} if the co-kernel of $[W_{\mathrm{dR}}^+ \to (W')_{\mathrm{dR}}^+]$ is free. A subobject (quotient) of $W$ is called \textbf{strict } (resp. \textbf{strict at $\sigma$}) if the inclusion (projection) is strict in the above sense (resp. strict at the component corresponding to $\sigma.$). 
	\end{defn}
	We recall for an $E$-$B$-pair $(W_e,W_{\mathrm{dR}})$ we can write (as $\Bdr^+ \otimes_{\QQ_p}E$-modules) $W_{\mathrm{dR}}^{(+)} = \prod_{\sigma \in \Sigma_E}(W_{\mathrm{dR},\sigma})^{(+)}$ by using the decomposition $E \otimes_{\QQ_p} \Bdr^+ \cong \prod_{\sigma \in \Sigma_E} \Bdr^+.$

	\subsection{Vector bundles on the Fargues--Fontaine curve}
	\label{sec:vectb}
	Let $X$ be a curve in the sense of \cite{farguesfontainecourbe}, i.e., a connected regular noetherian separated $1$-dimensional scheme together with a function $\deg \colon \abs{X} \to \NN.$ We denote by $K(X)$ the field of rational functions on $X,$ i.e., the stalk $\mathcal{O}_{x,\eta}$ in the generic point $\eta$ and by $\ord_x \colon K_X^\times \to \ZZ$ the valuation of $\mathcal{O}_{X,x}$ normalised such that the map $\ord_x$ is surjective.  
	We denote by $\Div(X)$ the free abelian group generated by $\abs{X},$ for $D=\sum_{x \in \abs{X}} a_x [x] \in \Div(X)$ we define $\deg(D):= \sum_{x \in \abs{X}} a_x \deg(x),$ and for $f \in K_X^\times$ let $\Div(f):= \sum_{x \in \abs{X}} \ord_x(f)[x].$
	We say $X$ is \textbf{complete} if $\deg(\Div(f)) = 0$ for all $f \in K(X)^\times.$ 
	If $X$ is a complete curve, then $\mathcal{O}_X(X)$ is a integrally closed subfield of $K_X$ (cf. \cite[Lemme 5.1.5]{farguesfontainecourbe}).
	We call $E:=\mathcal{O}_X(X)$ the \textbf{field of definition} of $X.$
	\subsubsection{Glueing vector bundles}
	
	\begin{defn}Let $S\subset \abs{X}$ be finite and $U:= X \setminus S.$
		
		We denote by $\mathcal{C}_S$ the category whose objects are triples $$(\mathcal{E}, (M_s)_{s \in S}, (u_s)_{s \in S})$$ consisting of 
		\begin{itemize}
			\item A vector bundle $\mathcal{E}$ on $U,$
			\item For each $s \in S$ a finite free $\mathcal{O}_{X,s}$-module $M_s,$
			\item For each $s \in S$ an isomorphism
			$$u_s \colon M_s \otimes_{\mathcal{O}_{X,s}} K_X \xrightarrow{\cong} \mathcal{E}_{\eta},$$
		\end{itemize}
		with the obvious morphisms. 
		Analogously we denote by $\widehat{\mathcal{C}}$ the category whose objects are triples $$(\mathcal{E}, (M_s)_{s \in S}, (u_s)_{s \in S})$$ consisting of 
		\begin{itemize}
			\item A vector bundle $\mathcal{E}$ on $U,$
			\item For each $s \in S$ a finite free $\widehat{\mathcal{O}_{X,s}}$-module $M_s,$
			\item For each $s \in S$ an isomorphism
			$$u_s \colon M_s \otimes_{\widehat{\mathcal{O}_{X,s}}}\widehat{K_{X,s}} \xrightarrow{\cong} \mathcal{E}_{\eta} \otimes_{K_X} \widehat{K_{X,s}},$$
		\end{itemize}
		We write $\mathcal{C}_S(X)$ to emphasize the dependence on $X$ if required. 
	\end{defn}
	with $\widehat{K_{X,s}}$ being the completion of $K_X$ with respect to the valuation $\ord_x$ (equivalently the field of fractions of the completion $\widehat{\mathcal{O}_{X,s}}$ of $\mathcal{O}_{X,s}).$
	\begin{prop}\label{prop:beauvillelazlo} (cf. \cite[Proposition 5.3.1]{farguesfontainecourbe})
		Let $S\subset \abs{X}$ be finite and $U:= X\setminus S.$
		For $s \in S$ and a vector bundle $\mathcal{E}$ on $X$ let $$c_s \colon \mathcal{E}(U) \otimes_{\mathcal{O}_{X,s}} K_X \to \mathcal{E}_\eta$$ be the natural map.
		\begin{enumerate}
			\item The functor 
			$$ \mathcal{E} \to (\mathcal{E}(U), (\mathcal{E}_{s})_{s}, (c_s)_s)$$ defines an equivalence between the category $\operatorname{Bun}_X$ of vector bundles on $X$ and $\mathcal{C}_S(X).$ 
			\item The functor 
			$$ \mathcal{E} \to (\mathcal{E}(U), (\widehat{\mathcal{E}_{s}})_{s}, (c_s \otimes \id_{\widehat{K}_{X,s}})_s$$ defines an equivalence between the category of vector bundles on $X$ and $\widehat{\mathcal{C}}_S(X).$ 
		\end{enumerate}
	\end{prop}
	\begin{proof}
		
		We proceed by induction on $\abs{S}.$
		If $S=\{s\}$ then the result is a consequence of the Beauville-Lazlo Theorem.
		For $\abs{S}>1$ we will explain the proof for $\mathcal{C}_S(X)$ the case $\widehat{C}_S(X)$ being similar. 
		For later use we remark that we can rephrase the result for $\mathcal{C}_S(X)$ as an equivalence 
		$$\operatorname{Bun}_X \cong \operatorname{Bun}_{X\setminus S} \times_{\operatorname{fMod}_{K_X}} \operatorname{fMod}_{\mathcal{O}_{X,s}},$$ where $\operatorname{fMod}_{R}$ denotes the category of finite free $R$-modules. 
		Suppose $\abs{S}>d$ and suppose that the theorem is true for every curve and every finite set of cardinality $<d.$  Let $s \in S.$ 
		The functor sending $\mathcal{E}$ to $(\mathcal{E}_{X\setminus \{s\}}, \mathcal{E}_s,c_s)$ defines, in particular, a vector bundle on the curve $Y:=X\setminus s.$ Applying the induction hypothesis to the curve $Y$ and the set $T = S\setminus \{s\}.$ we obtain that the functor sending $\mathcal{E}_{X\setminus{s}}$ to the tuple $(\mathcal{E}_{U}, (\mathcal{E}_t)_{t \in T}, (c_{t})_{t \in T})$ is an equivalence between the category
		$\operatorname{Bun}_{Y}$ and the category 
		$\mathcal{C}(Y)_T.$
		Consider the functor $\mathcal{C}_T(Y) \to \operatorname{fMod}_{K_X}$ sending a tuple $(\mathcal{F},M_t,u_t)$ to the $K_X$-vector space $\mathcal{F}_\eta.$
		We have an obvious forgetful functor $\mathcal{C}_S(X) \to \mathcal{C}_T(Y).$
		By unwinding the definition we have an equivalence of categories
		$\mathcal{C}_S(X) \cong \mathcal{C}_Y(T) \times_{\operatorname{fMod}_{K_X}} \operatorname{fMod}_{\mathcal{O}_{X,s}}.$
		Combining this with the equivalence
		$$\operatorname{Bun}_X \cong \operatorname{Bun}_{Y} \times_{\operatorname{fMod}_{K_X}} \operatorname{fMod}_{\mathcal{O}_{X,s}}$$ 
		yields the claim.

	\end{proof}
	\subsubsection{Equivariant vector bundles on the Fargues--Fontaine curve}

	In \cite{farguesfontainecourbe} Fargues and Fontaine attach to a perfectoid field $F$ of characteristic $p$ and a non-Archimedean local field $E$ a curve $X_{E,F}.$ 
	Explicitly 
	$X_{E,F} = \operatorname{Proj}(P_{E,\pi_E})$ with the graded ring \begin{equation}P_{{E,F},\pi_E} = \bigoplus_{n \in \NN_0} (B_{E,F}^+)^{\varphi_E = \pi_E^n}, \end{equation}
	where $B_{E,F}^+$ is the completion of $W(o_F)_{E}[1/p]$ with respect to the family of ``Gauß norms'' $\abs{-}_{\rho}$ with 
	\begin{align*}
		\abs{-}_\rho \colon W(o_F)_{E}[1/\pi_E,1/[\varpi_F]] &\to \RR_{\geq0}\\
		\sum_{k \gg -\infty} [a_k]\pi_E^k& \mapsto \sup_k \abs{a_k}\rho^k,
	\end{align*}
	$\varphi_E$ denotes the continuous extension of the $q = \abs{o_E/\pi_Eo_E}$-Frobenius on $W(F)_E,$ $\varpi_F$ is a pseudo-uniformiser of $F$ and $\rho \in [0,1)$. 
	For now let $P = \bigoplus_{n \in \NN_0} P_n$ be a graded ring and for $d>0$ we set $P^{(d)} := \bigoplus_{n\in \NN_0} P_{dn},$ which we view as a graded ring with $P_{dn}$ being the homogenous elements of degree $n.$
	
	As usual, we denote for $f \in P$ homogenous of degree $>0$ and $H \subset P$ a family of homogenous elements 
	\begin{itemize}
		\item The fundamental open subset $D_+(f):= \{\mathfrak{p} \in \operatorname{Proj}(P) \mid f \notin \mathfrak{p}\}.$
		\item The closed subsets $V_+(H):= \{\mathfrak{p} \in \operatorname{Proj}(P) \mid H\subseteq \mathfrak{p} \}.$
	\end{itemize}
	For later use we recall that 
	$D_+(fg) = D_+(f)\cap D_+(g)$ and that 
	$D_+(f)$ is isomorphic  to the affine scheme $\operatorname{Spec}(P[1/f]_0),$ where $P[1/f]_0$ denotes the $0$-graded part of $P[1/f]$ (cf. \cite[\href{https://stacks.math.columbia.edu/tag/00JP}{Tag 00JP,Tag 01MB}]{stacks-project}).
	\begin{thm} \label{thm:Xprop} Suppose $F$ is algebraically closed. 
		Let $X = X_{E,F}$ and $P = P_{E,F,\pi_E}.$ Then: 
		\begin{enumerate}
			\item Setting $\deg(x)=1$ for all $x \in \abs{X}$ turns $X$ into a complete curve with field of definition $E.$
			\item For every $t \in P_1 \setminus \{0\}$ the set $V_+(t)$ consists of a single closed point $\infty_t$ and the map $t \mapsto \infty_t$ defines a bijection 
			$$(P_1 \setminus \{0\})/E^\times \to \abs{X}$$
			\item For the standard open $D_+(t)=X \setminus \{\infty_t\}$ we have that the ring
			$\mathcal{O}_X(X \setminus \{\infty_t\}) = (P[1/t])_0 = (B^+_{E,F}[1/t])^{\varphi_E=1}$ is a PID. 
			\item The residue field $C_t$ at $\infty_t$ is complete, algebraically closed and we have a canonical isomorphism $C_t^\flat \cong F.$ In other words, $C_t$ is an untilt of $F.$
			\item The isomorphism $C_t^\flat \cong F$ induces an isomorphism 
			$$\widehat{\mathcal{O}_{X,\infty_t}} \cong \Bdr^+(C_t),$$ where $\Bdr^+(C_t)$ denotes the completion of $W(o_{C_t}^\flat)[1/\pi_L]$ for the topology induced by $$\ker(\theta_t \colon W(o_{C_t}^\flat)[1/\pi_L] \to C_t )$$
			
		\end{enumerate}
	\end{thm}
	\begin{proof}
		See  \cite[Théorème 6.5.2  6.5.2]{farguesfontainecourbe}. 
	\end{proof}
	Our case of interest will be $F = \CC_p^\flat$ and $E/\QQ_p$ finite. 
	We will henceforth write $X_{E}$ for $X_{\CC_p^\flat,E}.$
	
	By \cite[Théorème 6.5.2 ]{farguesfontainecourbe} there exists a canonical isomorphism 
	\begin{equation}
		\label{eq:basechangeFF} X_{F,E'} \cong E' \times_{E} X_{F,E} \end{equation} for any finite extension $E'/E.$
	It is obtained by applying the canonical isomorphism (for any graded ring $S$)
	\begin{equation}\operatorname{Proj}(S) \cong \operatorname{Proj}(S^{(d)}) \end{equation} together the isomorphism of graded algebras
	\begin{equation}E'\otimes_E P_{E,\pi_E} = P_{E',\pi_{E'}}^{([E:E'])}.\end{equation}
	Fargues and Fontaine show that for $E=\QQ_p$ the category of $G_K$-equivariant vector bundles on $X_{\QQ_p}$ is equivalent to the category of $G_K$-$B$-pairs. 
	More precisely they prove the following 
	\begin{prop} \label{prop:FFEquivalence}
		Let $\infty \in X_{\QQ_p}(\CC_p)$ be the closed point corresponding to $t_{cyc}.$
		We have 
		\begin{enumerate}
			\item For the completed stalk $\widehat{\mathcal{O}_{X,\infty}}$ we have $\widehat{\mathcal{O}_{X,\infty}} \cong \Bdr^+.$
			\item $U:=X \setminus \{\infty\} \cong \operatorname{Spec}(\fB_e).$
			\item The functor associating to a $G_K$-equivariant vector bundle $\mathcal{E}$ the pair $(\mathcal{E}(U), \widehat{\mathcal{E}_\infty})$ (together with the obvious glueing morphism over $\Bdr$) is an equivalence between the category of $G_K$-equivariant vector bundles on $X_{\QQ_p}$ and the category of $G_K$-$B$-pairs. 
		\end{enumerate}
	\end{prop}
	\begin{proof}
		See \cite[Section 10.1]{farguesfontainecourbe}.
	\end{proof}
	The following is a natural extension of the results of Fargues--Fontaine to the case where $E$ is not necessarily $\QQ_p.$ Similar considerations are used in \cite{pham2023prismaticfcrystalsanalyticcrystalline} in the analytic case. 
	\begin{prop} \label{prop:BVLE} Let $\{s_1,\dots,s_d\}= S\subset \abs{X_E}$ be $d$ distinct points, for each $i$ let $t_i \in P_1$ be a representative of $s_i.$ Let $U=X_E\setminus{S}.$ Let $\mathcal{E} \in \operatorname{Bun}_{X_E}.$
		\begin{enumerate}
			\item We have $B_{S,e,E}:=\mathcal{O}_X(U)= (B_{F,E}^+[1/(\prod_{i=1}^d t_i)])^{\varphi_E=1}.$
			\item $B_{S,e,E}$ is a PID and, in particular, $\operatorname{Bun}_U$ is equivalent to the category of finitely generated free $B_{S,e,E}$-modules and the category $\operatorname{Bun}_{X_E}$ is equivalent to the category of
			triples $$(N, (M_s)_{s \in S}, (u_s)_{s \in S})$$ consisting of 
			\begin{itemize}
				\item A finite free $B_{S,e,E}$-module $N,$
				\item For each $s \in S$ a finite free $\widehat{\mathcal{O}_{X,s}}$-module $M_s,$
				\item For each $s \in S$ an isomorphism
				$$u_s \colon M_s \otimes_{\widehat{\mathcal{O}_{X,s}}}\widehat{K_{X,s}} \xrightarrow{\cong} N \otimes_{B_{S,e,E}} \widehat{K_{X,s}},$$
			\end{itemize}
		\end{enumerate}
	\end{prop}
	\begin{proof}
		By Theorem \ref{thm:Xprop} We have $B_{S,e,E}= \bigcap_i D_+(t_i) = D_+(\prod t_i).$
		We thus obtain $\mathcal{O}_{X_E}(U) = P_{E}[1/\prod t_i]_0.$
		The second part follows from Proposition \ref{prop:beauvillelazlo}.
	\end{proof}
	\subsubsection{The elements $t_{\infty_\sigma}$ and $t_\sigma$} \label{subsec:elementst}
	In light of Proposition \ref{prop:BVLE} and Proposition \ref{prop:FFEquivalence}, we need to understand the (equivalence class in $P_1\setminus \{0\}/E^\times$ of the ) elements $t_{\infty_\sigma}$ corresponding to the preimage of $\infty \in X_{\QQ_p},$ which is in bijection with the embeddings $\sigma \colon E\to \CC_p.$ Explicitly, this bijection is given by sending an embedding $\sigma$ to the $\CC_p$-valued point $(\infty,\sigma)$ of $X_E \cong X_{\QQ_p} \times_{\QQ_p} E$ (cf. \cite[Theorem A.1]{pham2023prismaticfcrystalsanalyticcrystalline}).
	For a uniformiser $\pi_E$ of $E$ a natural base point is the element $t_{LT} \in B_E^+,$ which (up to units) is characterised by the fact that $G_E$ acts on $t_{LT}$ via the Lubin--Tate character $\chi_{LT} \colon G_E \to o_E^\times.$ One has $\varphi_q(t_{LT}) = \pi_E t_{LT}$ and hence $t_{LT}$ defines a point $\infty_{\id} \in X_E.$ Explicitly $t_{LT}$ is given as $t_{LT} = \log_{LT}(u),$ where $u$ is the modified Teichmüller lift of a generator $u=(u_n)_n$ of $\varprojlim LT[\pi_L^n]$ (cf. \cite[Section 2.1]{Schneider2017}). 
	\begin{defn} Let $E/\QQ_p$ be Galois. We let $\sigma \in \operatorname{Gal}(E/\QQ_p)$ act on $B^+_E= E\otimes_{E_0}B^+_{\QQ_p}$ via $\sigma \otimes \varphi_p^{v(\sigma)},$ where $v(\sigma)$ is the unique integer $v \in \{0,\dots,f-1\}$ such that $\sigma_{\mid E_0} = \varphi_p^v.$
		We denote by $t_{\sigma} \in B^+_E$ the element 
		$\sigma(t_{LT}),$ such that $t_{\id} = t_{LT}.$ 
	\end{defn}
	\begin{lem}  Let $E/\QQ_p$ be Galois and 
		let $t_N = \prod_{\sigma \in \operatorname{Gal}(E/\QQ_p)} t_{\sigma}.$
		\begin{enumerate}
			\item We have $\varphi_E(t_\sigma) = \sigma(\pi_E)t_\sigma$ and $g(t_\sigma) = \sigma(\chi_{LT}(g))t_{\sigma}$ for $g \in G_E.$
			\item There exists $x \in \breve{E}$ such that 
			$t_N = x t_{cyc}.$
		\end{enumerate}
	\end{lem}
	\begin{proof} See \cite[Proposition 3.4]{berger2021triangulable}.
	\end{proof}
	Since $\varphi_E(t_{cyc}) = p^{f(E/\QQ_p)}t_{cyc}$ one can deduce from the Lemma above that $\varphi_E(x) = \frac{ \operatorname{Norm}_{E/\QQ_p}(\pi_E)}{p^{f(E/\QQ_p)}}x.$ In the proof of \cite[Proposition 3.4]{Berger2013} Berger shows that (for $E/\QQ_p$ unramified) $v_p(\theta(t_{cyc}/t_{\id})) = \frac{1}{p-1}-\frac{1}{q-1}.$ We give a refinement of this statement for general $E.$ 
	To this end we recall some results, that appear implicitly in the work of Fourquaux. 
	\begin{lem} \label{lem:bergermateo} Let $\sigma \in \operatorname{Gal}(L/\QQ_p).$ Then the map $\theta$ sends $t_{\sigma}$ to $\log_{LT}^{\sigma}(x_\sigma),$ where $$x_{\sigma} := \lim_{n \to \infty} [\pi_L^n]^{\sigma} (u_n^{p^{v(\sigma)}}) \in \CC_p.$$
		We have 
		\begin{equation*}
			v_{p}(\log_{LT}^{\sigma}(x_\sigma)) = \begin{cases*}
				\frac{p^{v(\sigma)}}{e(q-1)}+\frac{1}{e}  & if  $v(\sigma)>0,$\\
				\frac{1}{e(q-1)}+v_p(\pi_L-\sigma(\pi_L))  & if $v(\sigma)=0.$
			\end{cases*}
		\end{equation*}
	\end{lem}
	\begin{proof}
		Since $t_{\sigma} = \log_{LT}^{\sigma} (\varphi_p^{v(\sigma)}u)$ the first part is just the definition of $\theta.$  The second part is \cite[Lemme 7]{fourquaux2009applications}.
	\end{proof}
	\begin{thm}
		Let $E/\QQ_p$ be Galois and $N_{E/\QQ_p}(\pi_E) =yp^f.$ Let $\Omega_E$ be a period as in \cite{schneider2001p} of $p$-valuation $\frac{1}{(p-1)}- \frac{1}{e(q-1)}.$  Let $\xi \in W(\overline{\kappa}_E)_E^\times$ be an element with $\varphi_q(\xi) = y^{-1}\xi$ then for 		
		$$\tilde{\Omega}:= \xi\prod_{\id \neq\sigma\in \Gal(E/\QQ_p)}t_\sigma$$
		we have  $0 \neq \theta(\tilde{\Omega}))/\Omega_E \in \CC_p^{G_E}$ and
		$$v_p(\theta(\tilde{\Omega}))= f-1 + \frac{1}{(p-1)}- \frac{1}{e(q-1)}+v_p(\mathcal{D}_{E/\QQ_p}),$$
		where $\mathcal{D}_{E/\QQ_p}$ denotes the different.
		In particular 
		$\theta(\tilde{\Omega}) \in p^{f-1+v_p(\mathcal{D}_{E/\QQ_p})} o_E^\times \Omega_E.$
	\end{thm} 	
	\begin{proof}
		By Lemma \ref{lem:bergermateo} we get that there exists $x \in E^\times$ such that $t_{cyc} = x \xi t_{N}$ and hence $\xi t_{N}/t_{\id}$ has Galois action given by $\tau= \chi_{cyc}/\chi_{LT}.$ As a consequence $\theta(\tilde{\Omega})/\Omega_E \in \CC_p$ is $G_E$-invariant and hence belongs to $E^\times.$ We now compute the valuation of $\theta(\tilde{\Omega})/\Omega_E.$ Since $\theta(\xi) \in o_{\breve{E}}^\times$ we just need to compute the valuation of $\theta(t_N/t_{\id}).$
		Let us denote $G= \Gal(E/\QQ_p)$ and let $I := \operatorname{Ker}(\Gal(E/\QQ_p) \to \Gal(\kappa_E/\mathbb{F}_p))$ such that $G/I \cong \operatorname{Gal}(E_0/\QQ_p).$
		Let us fix a $\tau \in G$ whose restriction to $\operatorname{Gal}(E_0/\QQ_p)$ is the arithmetic Frobenius. We have $$\operatorname{Gal}(E/\QQ_p) = \coprod_{\rho \in G/I = \operatorname{Gal}(E_0/\QQ_p)} I\sigma = \coprod_{i=0}^{h-1} I\tau^i.$$
		For every $\sigma \in I\tau^i$ we have $v(\sigma) = v(\tau^i) = i.$ Hence we get 
		\begin{align}
			v_p(\prod_{\sigma \neq \id} \log_{LT}^{\sigma}(x_\sigma)) & = \sum_{\rho \in I}\sum_{i=1}^{f-1} 	v_p( \log_{LT}^{\rho\tau^i}(x_{\rho \tau^i}))+ \sum_{\id \neq \sigma \in I} v_p( \log_{LT}^{\sigma}(x_\sigma)) \\
			&= \# I \sum_{i=1}^{f-1} 	(\frac{p^{i}}{e(q-1)}+\frac{1}{e} ) + \sum_{\id \neq \sigma \in I} 	\frac{1}{e(q-1)}+v_p(\pi_E-\sigma(\pi_E)) \\
			& = f-1 + \frac{p^f-1}{(p-1)(q-1)} - \frac{1}{q-1} + \frac{e-1}{e(q-1)} + v_p(\mathcal{D}_{E/E_0}) \\
			& = f-1 + \frac{1}{(p-1)}- \frac{1}{e(q-1)}+v_p(\mathcal{D}_{E/E_0}).
		\end{align}
		In the third equality we used $\# I =e$ and $p^f=q$ to simplify the expressions and we used that $E/E_0$ is totally ramified with uniformiser $\pi_E,$ which implies that $\mathcal{D}_{E/E_0}$ is generated by $\operatorname{Mipo}_{E_0}'(\pi_E) = \prod_{\id \neq \sigma \in I} (\pi_E-\sigma(\pi_E)).$ 
		Since $E_0/\QQ_p$ is unramified and $\mathcal{D}_{E/E_0}\mathcal{D}_{E_0/\QQ_p} = \mathcal{D}_{E/\QQ_p}$ we get the claim. 
	\end{proof}
	
	\begin{prop} \label{prop:elementetinftysigma}
		Let $\eta_\sigma \in o_{\breve{E}}^\times$ be  such that $\varphi_E(\eta_\sigma) = \frac{\pi_E}{\sigma(\pi_E)}\eta_\sigma.$ Let $t_{\infty_\sigma}:=\eta_\sigma t_\sigma.$ Then
		\begin{enumerate}
			\item $\varphi_E(t_{\infty_\sigma}) = \pi_E t_{\infty_\sigma}.$
			\item For $g\in G_E$ we have $gt_{\infty_\sigma} = \chi_{LT}^{\sigma}(g)\xi(g) t_{\infty_\sigma},$ with an unramified character $\xi(g).$
			\item The set $\{t_{\infty_\sigma} \mid \sigma \in \Sigma_E \} \subset P^1$ is a system of representatives corresponding to $p^{-1}(\infty)$ under the bijection from Theorem \ref{thm:Xprop} 2.).
		\end{enumerate}
	\end{prop}
	\begin{proof}
		The existence of $\eta_\sigma$ follows from a standard argument by dévissage using that $o_{\breve{E}}/\pi_E$ is algebraically closed. The first two points follow by definition. To see the third point, the target $p^{-1}(\infty)$ has exactly $[E:\QQ_p]$ elements hence it suffices to see that the $t_{\infty_\sigma}$ are not $E^\times$-multiples of each other. By \textit{(2)} each $t_{\infty_\sigma}$ is a period for a character with Hodge--Tate weight $1$ at the embedding $\sigma^{-1}$ and Hodge--Tate weight $0$ at the embeddings $\neq \sigma^{-1}.$ We conclude that the classes  $E^\times t_{\infty_{\sigma}} \in (P^1\setminus \{0\})/E^\times$ are distinct. 
	\end{proof}
	\begin{defn}
		Let $\mathcal{E}$ be a $G_K$-equivariant vector bundle on the Fargues--Fontaine curve $X_E$. Let $S = \{\infty_\sigma \mid \sigma \in \Sigma_E\} \subset X_E$ and $U:=X_E\setminus S.$ We define
		$W(\mathcal{E}):= (\mathcal{E}(U), \prod_{s \in S} (\widehat{\mathcal{E}_{s}})).$
	\end{defn}
	\begin{prop} \label{prop:bundletoBpair}
		Let $\mathcal{E}$ be a $G_K$-equivariant vector bundle on $X_E$. Let $S = \{\infty_\sigma \mid \sigma \in \Sigma_E\} \subset X_E$ and $U:=X_E\setminus S.$ Then
		\begin{enumerate}[(i)]
			\item $\mathcal{O}_{X_E}(U) = B^+_E[1/t_{cyc}]^{\varphi_E=1} \cong E \otimes_{\QQ_p}\Bcris^{\varphi_p=1}.$
			\item For every $s \in S$ there is a canonical $G_E$-equivariant isomorphism $\widehat{\mathcal{O}_{X_E,s}} \cong \widehat{\mathcal{O}_{X_{\QQ_p},\infty}} (\cong \Bdr^+).$
			\item Using the above isomorphisms $W(\mathcal{E})$ and the induced action of $G_K$ on $\mathcal{E}(U)$ and $\mathcal{E}_s$ turns $W(\mathcal{E})$ into a $G_K$-$E$-$B$-pair. 
		\end{enumerate}
	\end{prop}
	\begin{proof}
		For \textit{(i)} the first equation we note that by Proposition \ref{prop:BVLE} we have a priori $\mathcal{O}_{X_E}(U) = B_E^+[1/(\prod_\sigma{t_{\infty_\sigma}})]^{\varphi_E=1}.$ By Lemma \ref{lem:bergermateo} and Proposition \ref{prop:elementetinftysigma} the product in the denominator is a $\breve{E}^\times \subset (B_E^+)^\times$-multiple of $t_{cyc}.$ The isomorphism 
		$B_E^+[1/t_{cyc}]^{\varphi_E=1} \cong E \otimes \Bcris^{\varphi_p=1}$ is obtained by first using $\Bcris^{\varphi_p=1} = B_{\QQ_p}[1/t_{cyc}]^{\varphi_p=1}$  (cf. \cite[Théorème 6.5.2]{farguesfontainecourbe}) together with Galois descent for $E_0/\QQ_p,$ the maximal unramified subextension of $E/\QQ_p$, to obtain 
		$$(E_0 \otimes_{\QQ_p}B^+_{\QQ_p}[1/t_{cyc}])^{\varphi_{E_0}=1} \cong E_0 \otimes (B^+_{\QQ_p}[1/t_{cyc}])^{\varphi_{p}=1}$$ and lastly applying $E\otimes_{E_0}-$ to both sides (using that $\varphi_E= \id \otimes \varphi_{E_0}).$ For the second point we follow the argument in \cite[Proof of A.1]{pham2023prismaticfcrystalsanalyticcrystalline}. We have for every affine neighbourhood $\infty \in V \subset X_{\QQ_p}$ that $\mathcal{O}_{X_E}(p^{-1}(V)) \cong E \otimes_{\QQ_p} \mathcal{O}_{X_{\QQ_p}}(V),$ note also that the sections over $V$ are noetherian. This allows us to apply \cite[\href{https://stacks.math.columbia.edu/tag/07N9}{Lemma 07N9}]{stacks-project} to the
		finite ring map $ \mathcal{O}_{X_{\QQ_p}}(V) \to \mathcal{O}_{X_E}(p^{-1}(V))$ and the prime ideal corresponding to $\infty \in X_{\QQ_p}$ to obtain
		\begin{equation*}
			E \otimes_{\QQ_p} \widehat{\mathcal{O}_{X_{\QQ_p},\infty}} \cong \prod_{s \in S} \widehat{ \mathcal{O}_{X_E,s}}
		\end{equation*}
		while at the same time 
		$E \otimes_{\QQ_p} \widehat{\mathcal{O}_{X_{\QQ_p},\infty}} \cong E\otimes_{\QQ_p} \Bdr^+ \cong \prod_{\sigma \in \Sigma_E} \Bdr^+.$ To see that $W(\mathcal{E})$ defines an $E$-$B$-pair the only remaining point is the continuity of the $G_K$-action, this follows from \cite[Proposition 9.1.3]{farguesfontainecourbe}.
	\end{proof}
	\begin{thm} \label{thm:bundequiv} Let $E$ be Galois over $\QQ_p.$ The functor $\mathcal{E} \mapsto W(\mathcal{E})$
		Defines an equivalence between the category of $G_K$-equivariant vector bundles on $X_E$ category of $G_K$-$E$-$B$-pairs.
		The above equivalence restricts to an equivalence between $\operatorname{Rep}_E(G_K)$ and the category of $G_K$-equivariant vector bundles on $X_E$ of slope $0.$
	\end{thm}
	\begin{proof}
		Using Proposition \ref{prop:bundletoBpair} and Proposition \ref{prop:BVLE} we see that the data of an $E$-$B$-pair $W$ defines a vector bundle $\mathcal{E}(W)$ and the $G_K$-action translates to the bundle being $G_K$-equivariant.
		We will call $W(\mathcal{E})$ a $B$-pair although to be precise 
		one has to plug in the isomorphisms 
		$B_E^+[1/t_{cyc}]^{\varphi_E=1} \cong E \otimes_{\QQ_p}\fB_e$ and $\prod_{\sigma \colon E \to \Bdr^+}\Bdr^+\cong E\otimes_{\QQ_p}\Bdr^+$ to obtain a $B$-pair in the sense of Definition \ref{def:bpair}.
		
		One can check that they are inverse to each other. 
		Note that as a consequence of the classification of vector bundles a vector bundle $\mathcal{E}$ is of slope $0$ if and only if it is a direct sum of $\mathcal{O}_{X_E}.$ Furthermore  $H^0(X_E,\mathcal{O}_{X_E}(\lambda))$ is infinite dimensional if $\lambda >0,$ zero if $\lambda <0$ and $1$ if $\lambda=0$ (cf. \cite[Proposition 8.2.3]{farguesfontainecourbe}). Hence $\mathcal{E}$ is of slope $0$ if and only if $H^0(X_E,\mathcal{E})$ is  a $\operatorname{rank}(\mathcal{E})$-dimensional $E$-vector space. Let $V \in \operatorname{Rep}_E(G_K)$ and denote by $W(V):=(W_e,W_{\mathrm{dR}}^+)$ its $G_K$-$E$-$B$-pair. Then $V = W_e \cap W_{\mathrm{dR}}^+,$ where the intersection is formed in $W_{\mathrm{dR}}.$ Unwinding the constructions from \ref{prop:bundletoBpair} we conclude that $V=H^0(X_E,\mathcal{E}(W(V)))$ and hence that $\mathcal{E}(V)$ is semi-stable of slope $0.$ If conversely $\mathcal{E}$ is semi-stable of slope $0,$ then $\mathcal{O}_{X_E}\otimes_E H^0(X_E,\mathcal{E}) \to \mathcal{E}$ is an isomorphism. But then $W(\mathcal{E}) = (B_E^+[1/t_{cyc}]^{\varphi_E=1}\otimes_EV, \prod_{\sigma} \Bdr^+ \otimes_{E,\sigma} V)$ is the $B$-pair attached to $V=H^0(X_E,E).$ 
		
	\end{proof}
	From the construction it is easy to describe the de Rham objects. 
	\begin{rem}Let $W$ be the $B$-pair corresponding to $\mathcal{E}.$ Then the following are equivalent:
		\begin{enumerate}[(i)]
			\item $\widehat{\mathcal{E}_s}[1/t_{s}]$ admits a $G_K$-equivariant basis for every $s \in S.$ 
			\item $W$ is de Rham, i.e, $W_{\mathrm{dR}}$ admits a $G_K$-equivariant basis. 
			\item $\dim_KH^0(G_K,\widehat{\mathcal{E}_s}[1/t_{s}])=\operatorname{rank}(\mathcal{E})$ for all $s \in S.$ 
			\end{enumerate}
	\end{rem}

	\section{Crystalline vector bundles}
	In \cite{pham2023prismaticfcrystalsanalyticcrystalline} it is explained how to obtain a ``crystalline'' vector bundle on $X_E$ from a filtered $\varphi_q$-module. However, the notion of crystalline in loc.cit. agrees with the notion $E$-crystalline (crystalline and $E$-analytic). This is due to working over
	$\fB_E[1/t_{E}]^{\varphi_q=1},$ which is $\mathcal{O}_{X_E}(X\setminus{\{\infty_\sigma\} })$ for one fixed choice $\infty_\sigma$  above $\infty \in X_{\QQ_p}.$
	We explain how to realise $\varphi$-modules with $h = [E:\QQ_p]$-filtrations as vector bundles on the Fargues--Fontaine curve. These are related to multivariable $(\varphi,\Gamma)$-modules in the sense of Berger (cf. \cite{Berger2013}). 
	
	\begin{defn} Let $R$ be a ring $\varphi \colon R\to R$ an automorphism, $S$ a finite set, $R \to R'$ a ring extension. 
		We denote by $\operatorname{MF}^{\varphi}_{R,R',S},$ the category whose objects are finite free $R$-modules $D$ equipped with a $\varphi$-semi-linear automorphism and for each $s \in S$ an exhaustive increasing separated $\ZZ$-indexed filtration by $R'$-submodules $\operatorname{Fil}^i$ on $R'\otimes_RD .$ Equipped with the obvious notion of morphisms.
		
		An object of  $\operatorname{MF}^{\varphi}_{R,R',S}$ is called $\varphi$-module over $R$ with $\abs{S}$-filtrations. If $\abs{S}=1$ we omit $S$ from the notation. 
	\end{defn}
	\begin{ex}
		Let $V \in \operatorname{Rep}_E(G_K)$ be crystalline,
		then $D:=D_{cris}(V):= (\Bcris \otimes_{\QQ_p} V)^{G_K}$ is a finite free $K_0 \otimes_{\QQ_p}E$ module equipped and $\varphi_p \otimes \id$ induces an automorphism of $D_{cris}(V).$ 
		Furthermore, we have a natural filtration on $D_K:= K \otimes_{K_0} D = (K \otimes_{\QQ_p} E) \otimes_{K_0 \otimes_{\QQ_p}E} D = (\Bdr \otimes_{\QQ_p} V)^{G_K}$ induced by the $\ker(\theta)$-adic filtration on $\Bdr.$ We can view $D$ as an object in $\operatorname{MF}^{\varphi}_{K_0 \otimes_{\QQ_p}E,K \otimes_{\QQ_p}E}.$
		Alternatively we can use the decomposition 
		$$\Bcris \otimes_{\QQ_p} E \cong \prod_{i} \Bcris \otimes_{E_0,\varphi_p^i} E$$
		and hence $D = \bigoplus \varphi_p^i(D_0),$
		with $D_{0} = (\operatorname{B}_{cris}\otimes_{E_0}V)^{G_K},$ and the action of $\varphi_p$ on $D$ is uniquely determined by the action of $\varphi_q=\varphi_p^f$ on $D_0.$ Furthermore we have 
		$$K\otimes_{K_0}(K_0 \otimes_{\QQ_p}E) = \prod_{\Hom_{\QQ_p}(E_0,K)}\prod_{\Hom_{E_0}(E,K)} E\cong \prod_{\Hom_{\QQ_p}(E,K)}K.$$
		In other words, giving a filtration on $K \otimes_{K_0}D$ is equivalent to the data of $[E:\QQ_p]$-many filtrations on $K\otimes_{K_0\otimes_{E_0}E}D_0.$
		The argument provided shows that we have an equivalence of categories
		$$\operatorname{MF}^{\varphi_p \otimes \id_E }_{K_0 \otimes_{\QQ_p}E,K \otimes_{\QQ_p}E}\cong \operatorname{MF}^{\varphi_q}_{K_0 \otimes_{E_0}E,K,\Sigma_E},$$
		sending a filtered $(\varphi_p\otimes \id)$-module over $K_0 \otimes_{\QQ_p}E \cong \prod_{E_0 \to K} K_0\otimes_{E_0}E$ to the component $D_0$ at (a fixed) embedding $E_0 \to K$ and the $[E:\QQ_p]$-many filtrations obtained from the decomposition $K\otimes_{K_0}(K_0\otimes_{\QQ_p}E)  = \prod_{\sigma \in \Sigma_E}K.$ Meaning that we view $K \otimes_{K_0}D$ as a sum of isomorphic $K$-modules. The filtration on each summand defines (via the isomorphisms) a filtration on a fixed choice among the summands (equivalently a choice of embedding $E \to K$). 
	\end{ex}
	\subsubsection{Crystalline Bundles} Our next goal is to identify the crystalline vector bundles among $G_K$-equivariant bundles on $X_E.$
	The ``problem'' is that an $E$-linear representation is called crystalline, if it is crystalline as a $\QQ_p$-linear representation. When $E=\QQ_p$ the picture is very simple. A vector bundle $\mathcal{E}$ is crystalline if its restriction to $X_{\QQ_p}\setminus\{\infty\}$ is crystalline, which means that $\mathcal{E}(X\setminus\infty) \otimes_{\fB_e} B^+_{\QQ_p}[1/t_\infty]$ admits a $G_K$-invariant basis, or in other words  $D_{cris}(\mathcal{E}):=(\mathcal{E}(X\setminus\infty) \otimes_{\fB_e} B^+_{\QQ_p}[1/t_\infty])^{G_K}$ has $K_0$-dimension equal to the rank of $\mathcal{E}.$
	Set $$B_{e,S}:= H^0(X_E\setminus{S},\mathcal{O}_{X_E}) = (B_E^+[1/\prod_\sigma t_{\infty_\sigma}])^{\varphi_E=1},$$ which for $E=\QQ_p$ is the usual $\fB_e.$  
	Naively, we should replace $B_{\QQ_p}^+[1/t_\infty]$ by the ring 
	$B_E^+[1/t_\infty] = B_E^+[1/\prod_\sigma{t_{\infty_\sigma}}]$ (here we use the relationship between the $t_{\infty_\sigma}$ and $t_{cyc}$ from \ref{lem:bergermateo}). Unfortunately, the situation is slightly more delicate for the above mentioned reasons and we instead introduce the  ring $$B_S := B_{\QQ_p}^+[1/t_{\infty}] \otimes_{B^+_{\QQ_p}[1/t_\infty]^{\varphi_p=1}} B^+_E[1/\prod_{\sigma}(t_{\infty_\sigma})]^{\varphi_E=1}.$$ 
	The picture becomes clearer in the scheme theoretic language. We have (essentially by definition) a cartesian square 
	\[\begin{tikzcd}
		{\operatorname{Spec}(B_S)} & {X_E\setminus S} \\
		{\operatorname{Spec}(B_{\QQ_p}^+[1/t_\infty])} & {X_{\QQ_p}\setminus{\{\infty}\}}
		\arrow[from=1-1, to=1-2]
		\arrow[from=1-1, to=2-1]
		\arrow["p",from=1-2, to=2-2]
		\arrow[from=2-1, to=2-2]
	\end{tikzcd}\]
	Classicaly a vector bundle is crystalline if its pullback along the bottom map is the trivial representation. Similarly we will say that a vector bundle is crystalline if the restriction of scalars to $B^+_{\QQ_p}[1/t_\infty]$ of its pullback to $B_S$ is trivial. This boils down to the following definition.
	\begin{defn}
		A $G_K$ vector bundle $\mathcal{E}$ on $X_E$ is called \textbf{crystalline}, if the $G_K$-representation $\mathcal{E}(X\setminus{S})$ is \textbf{crystalline}, which by definition means, that 
		$B_{\QQ_p}^+[1/t_{\infty}] \otimes_{B^+_{\QQ_p}[1/t_\infty]^{\varphi_p=1}} \mathcal{E}(X\setminus{S})$ is trivial as a $B_{\QQ_p}^+[1/t_{\infty}]$ representation of $G_K.$ 
	\end{defn}
	\begin{rem} \label{rem:structure} We have $E \otimes_{\QQ_p}B_{\QQ_p}^+[1/t_{\infty}] = B_S$ and $B_S^{G_K}=E\otimes_{\QQ_p}K_0.$ 
	\end{rem}
	\begin{proof}
		Use $B_E^{\varphi_E=1} = E\otimes_{\QQ_p} \fB_e$ and hence 
		$$B_S = B_{\QQ_p}^+[1/t_{\infty}] \otimes_{\fB_e}\fB_e\otimes_{\QQ_p}E = E\otimes_{\QQ_p} B_{\QQ_p}^+[1/t_{\infty}].$$ The second part follows from the first using $$K_0\subset B_{\QQ_p}^+[1/t_{\infty}]^{G_K} \subset \operatorname{Frac}( B_{\QQ_p}^+)^{G_K}=K_0.$$
	\end{proof}
	Using the above Remark, we can equip $B_S$ with the $E$-linear Frobenius $\varphi_p\otimes \id,$ which induces $\varphi_p \otimes \id$ on $K_0 \otimes_{\QQ_p} E.$ 
	\begin{rem} \label{rem:cris} Let $\mathcal{E}$ be a $G_K$-bundle on $X_E.$ 
		We have that 
		$D_S(\mathcal{E}):=H^0(G_K,B_{\QQ_p}^+[1/t_{\infty}] \otimes_{B^+_{\QQ_p}[1/t_\infty]^{\varphi_p=1}} \mathcal{E}(X\setminus{S}))$ is free as a $K_0 \otimes_{\QQ_p}E$-module and the following are equivalent
		\begin{enumerate}[(i)]
			\item $\mathcal{E}$ is crystalline.
			\item $p_*\mathcal{E}$ is crystalline in the sense of \cite{farguesfontainecourbe}.
			\item $\operatorname{dim}_{K_0}(D_S(\mathcal{E}))=[E:\QQ_p]\operatorname{rank}(\mathcal{E}).$
			\item $\operatorname{rank}_{K_0\otimes_{\QQ_p}E}(D_S(\mathcal{E})) = \operatorname{rank}(\mathcal{E}).$
		\end{enumerate}
		
	\end{rem}
	\begin{proof}
		First of all, using the discussion after Remark \ref{rem:structure} we can view $D_S(\mathcal{E})$ as a finite $K_0\otimes_{\QQ_p}E$-module equipped with a $\varphi_p\otimes\id$ semilinear automorphism.
		By the same argument as in  \cite[Lemma 1.30]{nakamura2009classification}
		(with $K_0$ instead of $B^{\dagger}_{rig,K}$) we get freeness as a $K_0 \otimes_{\QQ_p}E$-module. The equivalence of \textit{(iii)} and \textit{(iv)} is clear using freeness. 
		The equivalence of \textit{(i)} and \textit{(ii)} follows by unwinding the definitions (note that $p_*$ is just restriction of scalars for quasi-coherent modules on affine schemes).
		The equivalence of \textit{(ii)} and \textit{(iii)} is \cite[Proposition 10.2.12]{farguesfontainecourbe}.
	\end{proof}
	Let us consider the pair of functors:
	\begin{align*}
		D_S\colon	\operatorname{Rep}_{B_{e,S}}(G_K) & \to \varphi_p\text{-}\operatorname{Mod}_{K_0\otimes_{\QQ_p}E} \\
		V& \mapsto (B_S \otimes_{B_{e,S}} V)^{G_K}
	\end{align*}
	\begin{align*}
		V_S \colon \varphi_p\text{-}\operatorname{Mod}_{K_0\otimes_{\QQ_p}E}  & \to 	\operatorname{Rep}_{B_{e,S}}(G_K) \\	
		D &\mapsto (B_S \otimes_{K_0\otimes_{\QQ_p} E} D)^{\varphi_p=1}.
	\end{align*}

	\begin{thm} \label{thm:cris}
		The functor $V_S$ is well-defined and fully faithful, $D_S$ is a right-adjoint. For each $M \in \operatorname{Rep}_{B_{e,S}}(G_K)$ there is a natural inclusion 
		$$V_S(D_S(M)) \hookrightarrow M,$$
		which is an isomorphism if and only if $\operatorname{rank}_{K_0\otimes_{\QQ_p}E}(D_S(M))= \operatorname{rank}_{B_{e,S}}(M).$
	\end{thm}
	\begin{proof}
		To see that $V_S$ is well defined note that $B_S^{(\varphi_p\otimes\id)=1} = \fB_e \otimes_{\QQ_p}E= B_{e,S}$ using Remark \ref{rem:structure}, Proposition \ref{prop:bundletoBpair} \textit{(i)} and Lemma \ref{lem:bergermateo}. 
		The remaining part of the proof works exactly as \cite[Proposition 10.2.12]{farguesfontainecourbe} by taking $R= B_S.$
	\end{proof}

	\begin{defn} Let $D \in \operatorname{MF}^{\varphi_q}_{K_0\otimes_{E_0}E,K,\Sigma_E}.$ Let $S = \{\infty_\sigma \mid \sigma \in \Sigma_E\}.$
		We denote by $\mathcal{V}(D)$ the $G_K$-equivariant vector bundle with 
		$$\mathcal{V}(D)(X\setminus S) = (B_E^+[1/\prod_\sigma t_{\infty_\sigma}]\otimes_{K_0\otimes_{E_0}E} D)^{\varphi_E=1}$$ 
		and $$\mathcal{V}(D)_{\infty_\sigma} = \operatorname{Fil}_\sigma^0(\Bdr \otimes_K D)$$ for $\sigma \in \Sigma_E.$
	\end{defn}
	
	\begin{prop} \label{prop:cryssimp}
		Let $D$ be a finite dimensional $K_0\otimes_{E_0}E$-vector space \footnote{Note that $K_0 \otimes_{E_0}E$ is indeed a field.} with a $\varphi_q\otimes\id$-semi-linear automorphism $\varphi_q.$ Then $\mathcal{V}(D)$ is crystalline. 
	\end{prop}
	\begin{proof}
		Let us write $K_0 \otimes_{\QQ_p}E = K_0\otimes_{E_0}E_0\otimes_{\QQ_p}E= \prod_{i=0}^{f-1} K_0\otimes_{\varphi_p^i,E_0} E$ and write $\operatorname{Ind}_{\varphi_q}^{\varphi_p}D:= E_0 \otimes_{\QQ_p}D = \prod D$ and define $\varphi_p\colon \operatorname{Ind}_{\varphi_q}^{\varphi_p}D \to \operatorname{Ind}_{\varphi_q}^{\varphi_p}D$ by setting $$(x_1,\dots,x_f) \mapsto (\varphi_E(x_f),x_1,\dots,x_{f-1}).$$ 
		We hence obtain a $\varphi_p\otimes\id$-semilinear endomorphism of $\operatorname{Ind}_{\varphi_q}^{\varphi_p}D.$ By Theorem \ref{thm:cris} and  Remark \ref{rem:cris} the $\fB_e$-representation 
		$$(B_S \otimes_{K_0\otimes_{\QQ_p}E}\operatorname{Ind}_{\varphi_q}^{\varphi_p}D)^{\varphi_p=1}$$ is crystalline. 
	We have isomorphisms \[K_0\otimes_{\QQ_p}E \cong E_0 \otimes_{\QQ_p}(K_0 \otimes_{E_0} E) \] and \[B_S = E\otimes_{\QQ_p}B_{\QQ_p}^+[1/t_\infty] \cong E_0 \otimes_{\QQ_p} E\otimes_{E_0}B_{\QQ_p}^+[1/t_\infty] \cong E_0 \otimes_{\QQ_p} B_E^+[1/t_\infty].\]
	Hence we have
$B_S \otimes_{K_0\otimes_{\QQ_p}E}\operatorname{Ind}_{\varphi_q}^{\varphi_p}D \cong E_0\otimes_{\QQ_p} (B_E^+[1/t_\infty] \otimes_{K_0\otimes_{E_0}E} D).$ If we equip the right hand side with the induced $\varphi_p$ action (as in the construction of $\operatorname{Ind}_{\varphi_q}^{\varphi_p}D$)
we can conclude $$(B_S \otimes_{K_0\otimes_{\QQ_p}E}\operatorname{Ind}_{\varphi_q}^{\varphi_p}D)^{\varphi_p=1} \cong (B_E^+[1/t_\infty] \otimes_{K_0\otimes_{E_0}E} D)^{\varphi_q=1}$$ by Shapiros Lemma. Because being crystalline only depends on the underlying $\fB_e$-representation, we see that $$\mathcal{V}(D)(X\setminus S) = (B_E^+[1/t_\infty] \otimes_{K_0\otimes_{E_0}E} D)^{\varphi_q=1}$$ is crystalline.
	\end{proof}
	For $E=K$ (i.e. $E_0=K_0,$ $K_0 \otimes_{E_0}E=E,$ and $\varphi_q\otimes \id_E= \id_E$) we give a less technical characterisation of crystalline bundles. 
	\begin{thm}\label{thm:equivalencecris}Let $\mathcal{E}$ be a $G_E$-equivariant vector bundle on $X_E.$ Let $S =  \{\infty_\sigma \mid \sigma \in \Sigma_E\}$ and $U:=X_E\setminus S.$
		
		The following are equivalent 
		\begin{enumerate}[(i)]
			\item $\mathcal{E}$ is crystalline.
			\item There exists a $\varphi_p\otimes \id$-module $\tilde{D}$ over $E_0 \otimes_{\QQ_p} E$ and $G_E$-equivariant isomorphism 
			$\mathcal{E}(U) = V_S(\tilde{D}) (= (B_S\otimes_{E_0 \otimes_{\QQ_p} E} \tilde{D})^{\varphi_p=1})),$  where $G_E$ acts trivially on $\tilde{D}.$
			
			\item There exists a $\varphi_{q}$-module $D$ over $E$ and $G_E$-equivariant isomorphism 
			$\mathcal{E}(U) =  \mathcal{V}(D)(=(B_E^+[1/\prod_\sigma t_{\infty_\sigma}]\otimes_E D)^{\varphi_q=1}),$  where $G_E$ acts trivially on $D.$
			
		\end{enumerate}
		\begin{proof} The equivalence of \textit{(i)} and \textit{(ii)} follows by combining Theorem \ref{thm:cris} and Remark \ref{rem:cris}. The implication
			\textit{(iii)} implies \textit{(ii)} is implicit in the proof of Proposition \ref{prop:cryssimp}, where $\tilde{D}$ is constructed given $D.$ Lastly it remains to show \textit{(ii)} implies \textit{(iii)}. To this end it suffices to show that a $\varphi_p\otimes \id$-module $\tilde{D}$ over $E_0\otimes_{\QQ_p} E$ is of the form $E_0\otimes_{\QQ_p} D$ for a suitable $D.$ This is just Galois descent for $E_0/\QQ_p.$
		\end{proof}
		
	\end{thm}
	\begin{thm} Consider the diagram of functors
		\[\begin{tikzcd}
			{\operatorname{Rep}_E^{E\text{-}\operatorname{cris}}(G_E)} & {\operatorname{Rep}_E^{\operatorname{cris}}(G_E)} & {\operatorname{Bun}^{\text{cris}}_{X_E}} \\
			{\operatorname{MF}^{\varphi_q}_{E,E,\{\id \}}} & {\operatorname{MF}^{\varphi_q}_{E,E,\Sigma_K}} & {\operatorname{Bun}^{\text{cris}}_{X_E}},
			\arrow[hook, from=1-1, to=1-2]
			\arrow["{D_{\text{cris},E}}", from=1-1, to=2-1]
			\arrow[hook, from=1-2, to=1-3]
			\arrow["{D_{\text{cris}}}", from=1-2, to=2-2]
			\arrow["{= }", no head, from=1-3, to=2-3]
			\arrow[hook, from=2-1, to=2-2]
			\arrow["{\mathcal{V}(-)}", from=2-2, to=2-3]
		\end{tikzcd}\]
		where we denote by abuse of notation $$D_{cris}\colon \operatorname{Rep}_E^{E\text{-}\operatorname{cris}}(G_E)\xrightarrow{D_{cris}}MF^{\varphi_p\otimes\id}_{E_0\otimes_{\QQ_p}E,E\otimes_{\QQ_p}E,\id}\cong MF^{\varphi_q}_{E,E,\Sigma_K},$$
		and
		the bottom left map is given by setting the filtration to be trivial at $\sigma \neq \id.$
		Then the diagram commutes and
		
		\begin{enumerate}
			\item The essential image of the left and middle vertical map consist of weakly admissible objects.
			\item 
			The functor $\mathcal{V}(-)$ is an equivalence of categories.
			\item The functor $\mathcal{V}(-)$ restricts to an equivalence between weakly admissible $\varphi_E$-modules with $[E:\QQ_p]$-filtrations and crystalline vector bundles of slope $0.$
			\item The essential image of the composite 
			$	{\operatorname{MF}^{\varphi_q}_{E,E,\{\id \}}} \to {\operatorname{Bun}^{\text{cris}}_{X_E}} $ 
			consists of bundles which are $E$-crystalline. 
		\end{enumerate}
		
	\end{thm}
	\begin{proof}
		The first part is ``weakly admissible implies admissible'' (resp. its analytic analogue from \cite{Kisin2009}).
		The second point is Theorem \ref{thm:equivalencecris}.
		The third point follows by combining \it{(ii)} and the equivalence Theorem \ref{thm:bundequiv}. For the last point note that the essential image of the bottom left map consists precisely of those objects such that the filtration is trivial at $\sigma \neq \id$ and compare with \cite{pham2023prismaticfcrystalsanalyticcrystalline}.
	\end{proof}
	\newpage

	\section{$\fB_e$-tuples}
	The goal of this section is to introduce yet another category equivalent to the category of vector bundles on $X_E$ for $E\neq \QQ_p.$ A category which we call $\fB_e$-tuples. The geometric intuition behind the construction is the following: 
	If $E=\QQ_p$ then complement of $$ \operatorname{Spec}(\fB_e)\to X_{\CC_p^\flat,\QQ_p}$$
	consists of a single point $x_\infty$ corresponding to $t_{cyc} \in (B_{\QQ_p}^+)^{\varphi=p}.$\\
	
	In general we have $[E:\QQ_p]$-points lying above $\infty \in X_{\CC_p^\flat,\QQ_p}$ with respect to $X_{\CC_p^\flat,E} \xrightarrow{f} X_{\CC_p^\flat,\QQ_p}.$
	We have already seen that, by looking at the section on the complement of the entire fibre  $f^{-1}(\infty)$ and keeping track of the (completed) stalks at all points, we get back the notion of $E$-$B$-pairs. 
	However, as soon as $[E:\QQ_p]>1,$ we can instead work with an open covering of $X_{\CC_p^\flat,E}$ by different punctured curves, which allows us to ``drop'' the $\Bdr^+$-part and instead work with multiple ``$\fB_e$-parts''.

	\begin{defn} Let $S =\{\infty_\sigma \mid \sigma \in \Sigma_E\}.$ 
		For $\emptyset \neq T \subseteq S$ let as before $B_{e,T}:= H^0(X_E \setminus T,\mathcal{O}_X) = B^+_E[\prod_{x \in T}\frac{1}{t_x}]^{\varphi_E=1}.$ Which, is the localisation of a PID hence itself a PID. 
		We say a family $ \mathfrak{L}$ of subsets of $S$ is a \textbf{co-covering}, if $\emptyset,S \notin \mathfrak{L}$ and for every $s \in S$ there exists some $T \in \mathfrak{L}$ such that $s \notin T.$
		A \textbf{$\fB_e$-tuple (indexed by $\mathfrak{L}$)} is a family $(M_T)_{T \in \mathfrak{L}}$ of free $B_{e,T}$-modules, together with isomorphisms
		$$ B_{e,T_1\cup T_2}  \otimes_{B_{e,T_1}} M_{T_1} \cong  B_{e,T_1\cup T_2}  \otimes_{B_{e,T_2}} M_{T_2}$$ for any pair $T_1,T_2 \in \mathfrak{L}$ satisfying the obvious cocycle condition whenever $(T_1 \cup T_2) \subseteq T_3  \in \mathfrak{L}.$
		A $G_K$-$\fB_e$-tuple is a $\fB_e$-tuple such that each $M_T$ carries a continuous semi-linear $G_K$-action. 
		We denote by $\mathrm{rank}(M)$ the $B_{e,T}$-rank of some (any) $M_T$ and call it the \textbf{rank of M}.
	\end{defn}
	\begin{rem}
		If $E=\QQ_p,$ then there exists no co-covering because $\abs{S}=1.$ 
	\end{rem}
	
	\begin{thm}
		\label{thm:equivalencetuples}
		Let $\mathfrak{L}$ be a co-covering of $S,$ then the functor 
		$$\operatorname{Bun}_{X_E} \to \{B_e\text{-tuples indexed by } \mathfrak{L}\}$$
		sending $\mathcal{F}$ to $(\mathcal{F}(X\setminus T)_T)$ is an equivalence of categories. 
		If $K/E$ is a finite extension then the same holds for the $G_K$-equivariant version. 
	\end{thm}
	\begin{proof}
		The inverse functor is given as follows: 
		Let $(M_T)_T$ be a $\fB_e$-tuple, then by definition each $M_T$ is a free $\mathcal{O}_X(X\setminus T)$-module. 
		The fact that $\mathfrak{L}$ is a co-covering, means that 
		$X_E\setminus T$ is an open affine subscheme and $(X_E\setminus T)_{T\in \mathfrak{L}}$ is a covering of $X_E.$ The isomorphisms 
		$$ B_{e,T\cup T'}  \otimes_{B_{e,T}} M_T \cong  B_{e,T\cup T'}  \otimes_{B_{e,T'}} M_{T'}$$ for any pair $T,T' \in \mathfrak{L}$ 
		translate to the fact that $X_E\setminus T \mapsto M_T$ defines a sheaf on $X_E,$ which is  locally free. 
	\end{proof}
	The smallest co-covering $\mathfrak{L}$ is given by two non-empty one-point sets $T_1,T_2$, whose intersection is empty.
	A more canonical choice is $\mathfrak{L}=\{  \{\infty_\sigma \mid \sigma \in \Sigma_E\}\}.$
	In this case $B_{e,\{\infty_\sigma\}} = B_E^+[1/t_{\infty_\sigma}]^{\varphi_E=1}.$
	\begin{rem}
		Let $\mathfrak{L}$ be a cocovering. Then 
		$$\bigcap_{T \in \mathfrak{L}} B_{e,T}=E,$$
		where the intersection is taken in $B_{e,\bigcup_{T\in \mathfrak{L}}T}.$
		In particular 
		$$B_E^+[1/t_{\infty_\sigma}]^{\varphi_E=1} \cap B_E^+[1/t_{\infty_\tau}]^{\varphi_E=1}=E$$ for any pair of embeddings $\sigma\neq \tau.$
		
	\end{rem}
	\begin{proof} Translate the fact that $H^0(X_{E},\mathcal{O}_{X_E})=E$ using the equivalence from Theorem \ref{thm:equivalencetuples}.
	\end{proof}
	\begin{rem}
		Let $(M_T)_{T \in \mathfrak{L}}$ be a $\fB_e$-tuple indexed by a co-covering $\mathfrak{L}$ let $\mathcal{V}$ be the corresponding vector bundle.
		\begin{enumerate}
			\item $\mathcal{V}$ is semi-stable of slope zero if and only if $\cap_{T \in \mathfrak{L}} M_T$ is a $\operatorname{rank}(M)$-dimensional $E$-vector space.
			\item All slopes of $\mathcal{V}$ are non-positive if and only if $\cap_{T \in \mathfrak{L}} M_T$ is finite dimensional.
		\end{enumerate}
	\end{rem}
	\begin{proof}
		This is a translation of \cite[Remarque 4.6]{farguesfontainecourbe} using that  $\cap_T M_T = H^0(X_E,\mathcal{V}).$
	\end{proof}
	\begin{rem} Fix an embedding $\sigma.$ Then for any $\tau \neq \sigma$ the image of the natural map
		$$B_E^+[1/t_{\infty_{\tau}}]^{\varphi_E=1} \xrightarrow{\iota_\sigma}\Bdr^+$$ induced by 
		$$E \otimes_{E_0} B^+_{\QQ_p} \xrightarrow{\sigma \otimes \varphi_p^{i(\sigma)}}  \Bdr^+$$ 
		is dense. 
	\end{rem}
	\begin{proof}
		Note that $\iota_{\sigma}(t_{\infty_\tau})\notin \operatorname{Fil}^1\Bdr^+,$ which makes the map well-defined.
		For every $x \in (B_E^+)^{\varphi_E=\pi^d}$ we have $\frac{x}{t_{\infty_\tau}} \in (B_E^+[1/t_{\infty_\tau}])^{\varphi_E=1}$ and $\iota_\sigma(\frac{x}{t_{\infty_\tau}})$ is invertible.
		We conclude that the image of $\iota_\sigma$ in $\Bdr^+/\operatorname{Fil}^i\Bdr^+$ contains $\iota_\sigma(t_{\infty_\tau})^{-1}(\operatorname{Im}((B_E^+)^{\varphi_E=\pi^d} \to \Bdr^+/\operatorname{Fil}^i\Bdr^+)).$  Since $\iota_\sigma(t_{\infty_\tau})$ is a unit, we obtain from the fundamental exact sequeunce that $\iota_\sigma$ is surjective modulo each $\operatorname{Fil}^i$ and hence that $\iota_\sigma$ has dense image. 
	\end{proof}
	\begin{defn}
		Let $\mathfrak{L}\subset S$ be a co-covering. For $T \subseteq X_E$ let $U_T:=X_{E} \setminus T$
		We denote by $C^\bullet(\mathfrak{L},\mathcal{O}_{X_E})$ be the Čech complex for the covering 
		$(U_T)_{T \in \mathfrak{L}},$ i.e.,
		$$\prod_{T \in \mathfrak{L}} \mathcal{O}_{X_E}(U_T) \to \prod_{T_1,T_2 \in \mathfrak{L}} \mathcal{O}_{X_E}(U_{T_1 \cup T_2}) \to \dots  .$$
	\end{defn}
	Note that by construction we have 	 $C^\bullet(\mathfrak{L},\mathcal{O}_{X_E}) = \mathbf R\Gamma(X_E,\mathcal{O}_{X_E}).$ In particular 
	$C^\bullet(\mathfrak{L},\mathcal{O}_{X_E}) \simeq E[0].$

	In order to extract Galois cohomology of representations (resp. $B$-pairs) in terms of $\fB_e$-tuples, we would like to show that the differentials are strict. The issue is, that the complex consists of LB spaces rather than Fréchet spaces and closed subspaces of LB-spaces (with the subspace topology) are not necessarily LB-spaces themselves. Instead we will use a result from the theory of almost $\mathbb{C}_p$-representations (cf. \cite{Fontaine2020}). Recall that a Banach representation $W$ of $G_K$ is called almost $\CC_p$-representation, if there exist Galois representation $V_1,V_2,$ $d \in \NN_0$ and embeddings $V_1 \to W, V_2 \to \CC_p^d$ such that 
	$W/V_1 \cong \CC_p^d/V_2.$ We denote by $\mathcal{C}(G_K)$ the category of almost $\CC_p$-representations. 
	\begin{defn}
		Let $W$ be a locally convex $E$-vector space. An admissible filtration on $W$ is a a filtration $F^nW$ indexed by $n\in \ZZ$ such that 
		\begin{enumerate}
			\item $\bigcup_{n \in \ZZ} F^nW = W$ and $\bigcap_{n\in \ZZ}W=0.$
			\item $F^nW/F^{n+r}W$ with the induced topology is a Banach space for every $n \in \ZZ$ and $r \in \NN_0.$
			\item $F^mW \cong \varprojlim_{r \geq 0} F^mW/F^{m+r}W$ as topological vector spaces.
			\item a $o_E$-submodule $U\subseteq W$ is open if and only if $U\cap F^nW$ is open in $W$ for all $n \in \ZZ.$
		\end{enumerate}
		In particular, $W$ is a strict LF space.
		We denote by $\widehat{\mathcal{C}}(G_K)$ the category of strict $\QQ_p$-LF spaces with continuous action of $G_K,$ admitting a $G_K$-equivariant admissible filtration such that $F^mW/F^{m+r}W \in \mathcal{C}(G_K)$ for every $m\in \ZZ,r \in \NN_0,$ with morphism continuous $G_K$-equivariant maps. 
	\end{defn}
	\begin{prop}  \label{prop:fontainecp} The category $\widehat{\mathcal{C}}(G_K)$ is abelian. 
		Any morphism $f \colon W_1 \to W_2$ in $\widehat{\mathcal{C}}(G_K)$ is strict. A sequence in $\widehat{\mathcal{C}}(G_K)$ is exact if and only if it is exact as a sequence of $\QQ_p$-vector spaces. The category $\mathcal{C}(G_K)$ is a Serre subcategory. 
	\end{prop}
	\begin{proof}
		This is \cite[Proposition 2.12]{Fontaine2020}.
	\end{proof}
	\begin{rem}
		The terms of the complex $C^\bullet(\mathfrak{L},\mathcal{O}_{X_E})$ with their natural topologies are objects of $\widehat{\mathcal{C}}(G_K).$ The differentials are strict.
	\end{rem}
	\begin{proof}
		
		The terms are finite products of spaces of the form
		$$B_E^+\left[\frac{1}{\prod_{s\in T} t_s}\right]^{\varphi=1} = \varinjlim_n \frac{1}{\prod_{s\in T} t_s}(B_E^+)^{\varphi=\pi^{\abs T}}.$$
		Hence they can be written as inductive limits of Banach spaces along closed immersions. More precisely, we can define an admissible $G_K$-equivariant filtration by $F^{-m}:= \bigcup_{n=0}^{m} \frac{1}{\prod_{s\in T} t_s}(B_E^+)^{\varphi=\pi^{\abs T}}.$ Using Proposition \ref{prop:fontainecp} and the fact that $\frac{1}{\prod_{s\in T} t_s}(B_E^+)^{\varphi=\pi^{\abs T}}$ belongs to $\mathcal{C}(G_K)$ we see that the graded pieces of the filtration belong to $\mathcal{C}(G_K)$
		The differentials are continuous and $G_K$-equivariant, hence using again Proposition \ref{prop:fontainecp} strict.
	\end{proof}

	\begin{thm}\label{thm:cohomology}
		Let $E\neq \QQ_p,$ let $\mathfrak{L}$ be a co-covering. Let $W$ be an $E$-$G_K$-$B$-pair. Let $\mathcal{V}$ be its corresponding vector bundle then
		$$
		[W_e \oplus W_{\mathrm{dR}}^+ \to W_{\mathrm{dR}}] 
		\simeq C^\bullet(\mathfrak{L},\mathcal{V}) 
		$$
		are strictly and $G_K$-equivariantly quasi-isomorphic.
		
	\end{thm}
	\begin{proof}
		As abstract complexes both compute the sheaf cohomology of $\mathcal{V}.$
		Because both complexes have terms in $\widehat{\mathcal{C}}(G_K),$ it suffices to show that there exists a continuous $G_K$-equivariant quasi-isomorphism between the two by Proposition \ref{prop:fontainecp} which asserts that the maps will automatically be strict. 
		First let us remark, that we can assume $\mathfrak{L} = \{T_1,T_2\}$ with $S= T_1 \coprod T_2.$
		Indeed, the covering obtained form $\mathfrak{L}$ can always be refined to the covering $\mathfrak{U}:=(X\setminus (S\setminus \{s\}),s\in S).$ But then any two $C^\bullet(\mathfrak{L},\mathcal{V}),C^\bullet(\mathfrak{L}',\mathcal{V})$ are (strictly) quasi-isomorphic (because transition maps between the two complexes are obviously maps in $\widehat{\mathcal{C}}(G_K)$). By the same reasoning, we can replace the complex $C^\bullet(\mathfrak{L},\mathcal{V}) $ by its alternating version $C^\bullet_{alt}(\mathfrak{L},\mathcal{V}).$ Without loss of generality we can fix the co-covering $\mathfrak{L}$ as above. 
		We identify $S = \{\infty_\sigma, \sigma \in \Sigma_E\}$ with the set of embeddings $\Sigma_E.$
		Let us write $\mathcal{V}_{e,T_i}:=\mathcal{V}(X\setminus T_i)$ and $\mathcal{V}_e:=\mathcal{V}(X\setminus S).$ 
		We write $\iota_{T_i} \colon \mathcal{V}_e \cong W_e \to \prod_{\sigma \in T_i} W_{\mathrm{dR}}$ for the natural map $W_e \to W_{\mathrm{dR}}$ followed by the projection 
		$$W_\mathrm{dR} = \prod_{\sigma \in S} W_{\mathrm{dR},\sigma} \to \prod_{\sigma \in T_i} W_{\mathrm{dR,\sigma}}$$
		By Proposition \ref{prop:bundletoBpair} we can identify 
		$\mathcal{V}_e \cong W_e$ and $W_{\mathrm{dR},\sigma}$ with the completed stalk  $\widehat{\mathcal{V}}_{\infty_\sigma}.$
		Now consider the following diagram
		\[\begin{tikzcd}
			& 0 & 0 \\
			& {H^0(X_{E},\mathcal{V})} & {H^0(X_{E},\mathcal{V})} \\
			{C^\bullet_{alt}(\mathfrak{L},\mathcal{V})\colon} & {\mathcal{V}_{e,T_1} \oplus \mathcal{V}_{e,T_2}} & {\mathcal{V}_{e}} \\
			{C^\bullet(W):} & {W_e} & {W_{\mathrm{dR}}/W^+_{\mathrm{dR}}} \\
			& {H^1(X_E,\mathcal{V})} & {H^1(X_E,\mathcal{V})} \\
			& 0 & 0
			\arrow[from=1-2, to=2-2]
			\arrow[from=1-3, to=2-3]
			\arrow["{\cdot 2}", from=2-2, to=2-3]
			\arrow["{(x,-x)}", from=2-2, to=3-2]
			\arrow["{\operatorname{res}}", from=2-3, to=3-3]
			\arrow["{\iota_1-\iota_2}", from=3-2, to=3-3]
			\arrow["{\iota_1+\iota_2}"', from=3-2, to=4-2]
			\arrow["{\iota_{T_1}-\iota_{T_2}}", from=3-3, to=4-3]
			\arrow["{\iota_S}"', from=4-2, to=4-3]
			\arrow[from=4-2, to=5-2]
			\arrow[from=4-3, to=5-3]
			\arrow["{\operatorname{id}}"', from=5-2, to=5-3]
			\arrow[from=5-2, to=6-2]
			\arrow[from=5-3, to=6-3]
		\end{tikzcd}\]
		Where by abuse of notation we use the same symbols for 
		$$\iota_i \colon \mathcal{V}_{e,T_i} \xrightarrow{\operatorname{res}}\mathcal{V}_{e,S}$$ and the induced maps 
		$$\iota_i \colon \mathcal{V}_{e,T_i} \xrightarrow{\operatorname{res}}\mathcal{V}_{e,S} \cong W_e.$$
		The columns are exact and tracing through the identifications made and using that the image of $\mathcal{V}_{e,T_i}$ in $W_{\mathrm{dR},\sigma}/W_{\mathrm{dR},\sigma}^+$ is zero whenever $\sigma \notin T_i$, we can see that the diagram commutes, i.e., the middle square induces a map of complexes $C^\bullet_{alt}(\mathfrak{L},\mathcal{V}) \to C^\bullet(W)$ with acyclic kernel and co-kernel (given by the top and bottom row of the diagram). With this explicit description, one can check that the maps are continuous and $G_K$-equivariant.
		
	\end{proof}
	
	It follows from Theorem \ref{thm:cohomology} that $\mathbf{R}\Gamma_{cts}(G_K,W)$ can be computed 
	using the complex $C^\bullet(\mathfrak{L},\mathcal{V})$ by taking the total complex of the continuous co-chain double complex, which (due to strictness of differentials) can be viewed as the evaluation at a point of the corresponding condensed group cohomology which can be defined as a derived functor. 
It would be interesting to give a more explicit description  using the abelian category $\mathcal{C}(G_K)$ and the larger category $\widehat{\mathcal{C}}(G_K).$ For an abelian category $\mathcal{A}$ and $X,Y \in \mathbf{D}(\mathcal{A})$ we write $\operatorname{Ext}^i_{\mathcal{A}}(X,Y):=\operatorname{Hom}_{\mathbf{D}(\mathcal{A})}(X[-i],Y).$ We denote by $\mathbf{D}_{C(G_K)}(\widehat{\mathcal{C}}(G_K))$ the full subcategory of $\mathbf{D}(\widehat{\mathcal{C}}(G_K))$ consisting of objects whose cohomology belongs to ${\mathcal{C}}(G_K).$
	Fontaine showed that for objects $Z \in \mathcal{C}(G_K)$ we have 
	$$\operatorname{Ext}^i_{\mathcal{C}(G_K)}(\QQ_p,Z)\cong H^i_{cts}(G_K,Z)$$ (cf. \cite[Proposition 6.7]{Fontaine2003}).
	We do not know whether the natural functor 
	$\mathbf{D}_{C(G_K)}(\mathcal{C}(G_K)) \to \mathbf{D}(\widehat{\mathcal{C}}(G_K))$ is fully faithful. 
Because $\mathcal{C}(G_K)$ is a Serre subcategory, one has
$$\operatorname{Ext}^i_{\mathcal{C}(G_K)}(\QQ_p,Z) = \operatorname{Ext}^i_{\widehat{\mathcal{C}}(G_K)}(\QQ_p,Z)$$ for $i=0,1$ for objects $Z \in \mathcal{C}(G_K).$ Using effaceability (cf. proofs of \cite[Proposition 6.7 and Proposition 6.8]{Fontaine2003}) and the Snake Lemma one can show that the map \begin{equation}\label{eq:ext2}\operatorname{Ext}^2_{\mathcal{C}(G_K)}(\QQ_p,Z) \to \operatorname{Ext}^2_{{\widehat{\mathcal{C}}(G_K)}}(\QQ_p,Z)\end{equation} is injective, but we do not know if it is an isomorphism. Using Harder--Narasimhan filtrations one can write $C^\bullet(W)$ as a sequeunce of distinguished triangles in $ \mathbf{D}(\widehat{\mathcal{C}}(G_K))$ concentrated in a single degree. Hence if \eqref{eq:ext2} were an isomorphism for all objects, then one would also have $H^i(G_K,W) \cong \operatorname{Ext}^i_{\widehat{\mathcal{C}}(G_K)}(\QQ_p,C^\bullet(W)).$

	\subsection{Berger's multivariable theory} To relate to the situation of Berger let us assume $K=E=E_0$ is unramified.
	In \cite{Berger2013} Berger constructs a bijection on objects: 
	$D \mapsto M(D)$ between the category of $\varphi_q$-modules $D$ over $E$ with $[E:\QQ_p]$-filtrations and $(\varphi_q,\Gamma_E^{LT})$-modules over the Robba ring $\mathcal{R}(\underline{Y}) = \mathcal{R}_{E}(Y_1,\dots,Y_f)$ with in $[E:\QQ_p]$-variables with coefficients in $E.$ It follows from the construction (\cite[Theorem 5.2]{Berger2013}) that $M(D)$ is the base change of a (reflexive, coadmissible) $\cR^+_E(\underline{Y})$-module, which should be seen as a multivariable analogue of the Wach module of a crystalline $(\varphi,\Gamma)$-module. 
	\begin{rem}
		If $E/\QQ_p$ is unramified then for the uniformiser $\pi_E=p$ the elements $t_{\sigma}$ and $t_{\infty_\sigma}$ agree and if $\sigma = \varphi_p^i,$ then 
		$t_{\sigma} = \varphi_p^i(t_{E}).$
	\end{rem}
	By sending $Y_i$ to $\varphi_p^i(u)$ (with $u$ as before in Section \ref{subsec:elementst}) one gets an embedding (cf.\textit{(loc.cit)}) $\cR^+_E(\underline{Y}) \to  B^+_{\QQ_p} \footnote{In the notation of {\textit{(loc.cit.)}} $B_{\QQ_p}^+$ is called $\widetilde{\fB}^+_\text{rig}.$}.$ This can be extended to 
	$$\cR_E(\underline{Y}) \to \tilde{\fB}^{\dagger}_{\text{rig}}$$
	\begin{defn} Let $E/\QQ_p$ be unramified
		let $D \in \operatorname{MF}^{\varphi_q}_{E,E,\Sigma_E}$ and let $M:=M^+(D).$ Let $x_i$ be the point corresponding to the embedding $\varphi_p^i.$
		Define $V_M$ to be the equivariant vector bundle 
		with $$V_M(X\setminus{S}) = (B_{\QQ_p}^+[1/t]\otimes_{\cR^+_E(\underline{Y})}M^+(D))^{\varphi_q=1}.$$
		and for $i=1,\dots,f-1$ $$\widehat{(V_M)_{x_i}} = \Bdr^+ \otimes^{\varphi_p^i}_{\cR_E^{[s,1)}(\underline{Y})} M^{[s,1)} \text{ for } s \text{ large enough}.$$
		
	\end{defn}
	
	Unwinding the definitions one can show the following Remark.
	\begin{rem}
		We have $V_{M^+(D)} = \mathcal{V}(D).$
	\end{rem}

	Using the notion of $\fB_e$-tuples we can extend the definition of $\mathcal{V}(M)$ without worrying about embeddings into $\Bdr.$
	\begin{prop}
		Let $M$ be a not necessarily free $(\varphi_q,\Gamma_E^{LT})$-module over $\mathcal{R}_E(Y_1,\dots,Y_f)$ such that 
		$M^{(i)}:=M[1/\log_{LT}(Y_i)]\otimes_{\mathcal{R}_E} \tilde{\fB}^{\dagger}_{\text{rig}}  (=M \otimes_{\mathcal{R}_E}\tilde{\fB}^{\dagger}_{\text{rig}}[1/\varphi_p^i(t_{LT}))])$ is free of rank $d$ for every $i.$
		Then 
		$V_{{\varphi_p^i}}:= (M^{(i)})^{\varphi_q=1}$ defines a $\fB_e$-tuple indexed by the co-covering $\{\{\varphi_p^i \}\mid i=1,\dots,f-1\}.$
	\end{prop}
	\begin{proof}
		By assumption $M^{(i)}$ is a $(\varphi_q,G_E)$-module over $\tilde{\fB}^{\dagger}_{\text{rig}}[1/\varphi_p^i(t_{LT})].$ It remains to see that it admits a $\varphi_p^f$-stable basis. The argument of the proof of \cite[Theorem 6.11]{Berger2013} implies that $N:=M \otimes_{\mathcal{R}_E} \tilde{\mathbf{B}}^{\dagger}_{\text{rig}}$ is free of rank $d.$ From here one can deduce the existence of a $\varphi$-invariant basis for $N[1/(\varphi_p^i(t_{LT}))]$ by using the Dieudonné--Manin classification of $\varphi$-modules over $\tilde{\fB}^{\dagger}_{\text{rig}}.$ To this end one writes $N$ as a sum of standard modules and checks it for each summand by using $ \varphi_q(\varphi_p^i(t_{LT}))=p\varphi_p^i(t_{LT})$ for any $i$ (cf. \cite[Proposition 2.2.6]{berger2008construction} in the cyclotomic case).
	\end{proof}

	\let\stdthebibliography\thebibliography
	\let\stdendthebibliography\endthebibliography
	\renewenvironment*{thebibliography}[1]{%
		\stdthebibliography{BHH+22}}
	{\stdendthebibliography}

	\bibliographystyle{amsalpha}
	
	\bibliography{Literatur}
\end{document}